\documentclass[12pt,letterpaper]{amsart}
\usepackage{amsmath, amsthm, amssymb, xspace, mathrsfs}
\usepackage[tmargin=1.4in,bmargin=1.2in,rmargin=1.4in,lmargin=1.4in]{geometry}
\usepackage[breaklinks=true]{hyperref}
\usepackage{amsmath,amscd}
\usepackage{tikz-cd}

\theoremstyle{plain}
\newtheorem{theorem}{Theorem}[section]
\newtheorem{lemma}{Lemma}[section]

\theoremstyle{definition}

\newcommand{\mf}[1]{\displaystyle{\mathfrak{#1}}}

\newcommand{\comment}[1]{}

\DeclareMathOperator{\spec}{\ensuremath{Spec}}

\DeclareMathOperator{\Gr}{\ensuremath{gr}}

\begin{document}
\title{Generic simplicity of quantum Hamiltonian reductions}
\author{Akaki Tikaradze}
\email{Akaki.Tikaradze@utoledo.edu}
\address{University of Toledo, Department of Mathematics \& Statistics, 
Toledo, OH 43606, USA}
\begin{abstract}

Let a reductive group $G$ act on a smooth affine complex algebraic variety $X.$
Let $\mathfrak{g}$ be the Lie algebra of $G$ and $\mu:T^*(X)\to \mathfrak{g}^*$ be the moment map.
If the moment map is flat, and for a generic character $\chi:\mathfrak{g}\to\mathbb{C}$, the action of $G$ on
$\mu^{-1}(\chi)$ is free, then we show that for very generic characters $\chi$ the corresponding quantum Hamiltonian reduction
of the ring of differential operators $D(X)$ is simple. 
\end{abstract}

\maketitle

Let a reductive algebraic group $G$ act on a smooth affine algebraic variety $X$ over $\mathbb{C}.$
Let $\mathfrak{g}$ be the Lie algebra of $G.$
Let $\mu:T^*(X)\to \mathfrak{g}^*$ be the corresponding moment map.
We will assume that this map is flat, and for generic $G$-invariant character $\chi\in \mathfrak{g}^*$ the action of $G$
on $\mu^{-1}(\chi)$ is free. 
   
       Given a $G$-invariant character $\chi\in \mathfrak{g}^*$, denote by $U_{\chi}(G, X)$ the quantum Hamiltonian reduction of
       $D(X)$ with respect to $\chi$. So, 
       $$U_{\chi}(G, X)=(D(X)/D(X)\mathfrak{g}^{\chi})^G,$$ where $\mathfrak{g}^{\chi}=\lbrace g-\chi(g)\in D(X), g\in\mathfrak{g}\rbrace.$
The usual filtration on $D(X)$ by the order of differential operators induces the corresponding filtration on $U_{\chi}(G, X).$
Then it follows from the flatness of the moment map that 
$$\Gr  U_{\chi}(G, X)=\mathcal{O}(\mu^{-1}(0)//G).$$
In what follows by a very generic subset we mean a complement of a union of countably
many proper closed Zariski subsets.
Under these assumptions we have the following result.

\begin{theorem}\label{main}
For very generic values of a $G$-invariant character $\chi\in \mathfrak{g}^*,$ the corresponding
quantum Hamiltonian reduction $U_{\chi}(G, X)$ is simple.
Moreover, if $f\in \mathfrak{g}_{\mathbb{Z}}/[\mathfrak{g}_{\mathbb{Z}}, \mathfrak{g}_{\mathbb{Z}}]$ is so that $G$ acts freely on $\mu^{-1}(\chi)$ whenever
$\chi(f)\neq 0$, then $U_{\chi}(G, X)$ is simple for all $\chi$ such that $\chi(f)\notin \mathbb{Q}.$
\end{theorem}

The proof is be based on the reduction modulo $p^n$ technique for a large prime $p.$

At first, we recall that given a ring $R$ such that $p$ is not a zero divisor, then the center
of its reduction modulo $p, R_p=R/pR$ acquires a natural Poisson bracket, to be referred to as the reduction
modulo $p$ Poisson bracket, defined  as follows. Given central elements $x, y\in Z(R_p),$ let $x', y'\in R$ be their lifts.
Then $$\lbrace x, y\rbrace =(\frac{1}{p}[x', y'])\mod p\in Z(R_p).$$

We  use the following result [\cite{T}, Corollary 8].

\begin{lemma}\label{simple}
Let $\bf{k}$ be a perfect field of characteristic $p.$
Let $A$ be a  $p$-adically complete topologically free $W(\bold{k})$-algebra,
such that $A_1=A/pA$ is an Azumaya algebra over its center $Z_1.$ Assume that $\spec(Z_1)$
is a smooth symplectic $\bold{k}$-variety under the reduction modulo $p$  Poisson bracket.
Then $A[p^{-1}]$ is topologically simple.
\end{lemma}

Next  we need to recall some results and notations associated with quantum Hamiltonian reduction of the ring of crystalline
differential operators in characteristic $p$ from \cite{BFG}. 

Let $X$ be a smooth affine variety over an algebraically closed field $\bold{k}$ of characteristic $p$, 
and $G$ be a reductive algebraic group over $\bold{k}$ with the Lie algebra $\mf{g}.$
Denote by $D(X)$  the ring of crystalline differential operators on $X.$
As before, we have the moment map $\mu:T^*(X)\to\mathfrak{g}^*$ and the algebra homomorphism  $U(\mf{g})\to D(X).$
Now recall that the $p$-center of $U(\mf{g})$, denoted by $Z_p(\mathfrak{g}),$ is generated by $g^p-g^{[p]}, g\in \mf{g}.$
We get an isomorphism
$$i:\text{Sym}(\mf{g})^{(1)}\to Z_p(\mathfrak{g}).$$
 On the other hand,  the center of $D(X)$ 
is generated by $\mathcal{O}(X)^p$ and $\xi^p-\xi^{[p]}, \xi\in T_X$ and this leads to an isomorphism
$$\mathcal{O}(T^*(X))^{(1)}\to Z(D(X)).$$
 We have $\eta':Z_p(\mf{g})\to Z(D(X))$ and the corresponding homomorphism
$$\eta:\text{Sym}(\mf{g})^{(1)}\to \mathcal{O}(T^*(X))^{(1)}.$$
Given $\chi\in \mathfrak{g}^*$, then $\chi^{[1]}\in \mathfrak{g}^*$ is defined as follows: 
$$\chi^{[1]}(g)=\chi(g)^p-\chi(g^{[p]}),\quad g\in \mathfrak{g}.$$
Using the above homomorphisms it follows that the center of $U_{\chi}(G, X)$ contains $\mathcal{O}(\mu^{-1}(\chi^{[1]})//G).$
In this setting the following holds.
\begin{lemma}\label{Azumaya}\cite{BFG}
Let $\chi\in (\mathfrak{g}^*)^G$ be a character.  
Then $U_{\chi}(G, X)$ is a finite algebra over  $\mu^{-1}(\chi^{[1]})//G.$ If $G$ acts freely of $\mu^{-1}(\chi^{[1]}),$
then $U_{\chi}(G, X)$ is an Azumaya algebra over $\mu^{-1}(\chi^{[1]})//G.$

\end{lemma}
We need the following criterion of simplicity of certain filtered quantizations.

\begin{lemma}\label{criterion}
Let $S\subset\mathbb{C}$ be a finitely generated ring, and let $R$ be a filtered
$S$-algebra, such that $\Gr(R)$ is a finitely generated commutative ring over $S.$
Assume that for all  large enough primes $p$ the algebra $R_p=R/pR$ is an Azumaya algebra over its center $Z_p$, moreover
$\spec(Z_p)$ is a smooth symplectic variety over $S_p$ under the reduction modulo $p$ Poisson
bracket. Let $F$ be the field of fractions of $S.$ Then $R_F=R\otimes_SF$ is a simple ring.
\end{lemma}

\begin{proof}
Let $I$ be a nonzero two sided ideal of $R$ such that $(R/I)_F\neq 0.$ 
After localizing $S$ further, we may assume  using the generic flatness theorem 
that $\Gr(R/I)$ and $R/I, R$ are free $S$-modules. Hence for $p\gg 0,$
$\bar{I}_p$ (the $p$-adic completion of $I$) is a topologically free nontrivial
two-sided ideal of $\bar{R}_p$ (the $p$-adic completion of $R$). Now Lemma \ref{simple} yields a contradiction.

\end{proof}
Next we state a result implying that taking quantum Hamiltonian reduction and reducing modulo a large prime commute.
The statement and its proof were kindly provided by W. van der Kallen ( via mathoverflow.org.)
Possible mistakes in the proof below are solely due to the author.

\begin{theorem}[van der Kallen]\label{change}
Let $S$ be a commutative Noetherian ring of finite homological dimension, let $R$ be a commutative $S$-algebra
flat over $S.$ Let $G$ be a split  reductive group over $S$ acting on $R.$ 
Then for all $p\gg 0$ and a base change to a characteristic $p$ field $S\to\bold{k},$ the map
$R^G\otimes_S\bold{k}\to R_{\bold{k}}^{G_{\bold{k}}}$ is  surjective.
\end{theorem}
\begin{proof}
At first, recall that there exists an integer $n\geq1$ so that $H^i(G, S[\frac{1}{n}])=0$ for all $i$
[\cite{FW}, Theorem 33].
This implies $H^i(G, S[\frac{1}{n}]\otimes_SN)=0$ for any  $S$-module $N$ with the trivial $G$-action (since $S$ has a finite global dimension).
Let $D, N$ be respectively the image  and kernel of  the map $S[\frac{1}{n}]\to \bold{k}.$ As $H^1(G, S[\frac{1}{n}]\otimes_SN)=0,$
we get that $(R\otimes_SS[\frac{1}{n}])^G\to (R\otimes_SD)^G$ is surjective. Now flatness of $\bold{k}$  over $D$ yields that
$$(R\otimes_SD)^G\otimes_D\bold{k}=R_{\bold{k}}^{G_{\bold{k}}}.$$
Therefore, we obtain the desired surjectivity $R^G\otimes_S\bold{k}\to R_{\bold{k}}^{G_{\bold{k}}}$
\end{proof}

\begin{proof}[Proof of Theorem \ref{main}]

Recall that  some $0\neq f\in \mathcal{O}((\mathfrak{g}^*)^G)$ has the property that 
for any $\chi\in (\mf{g}/[\mf{g}, \mf{g}])^*$ such that $f(\chi)\neq 0,$
the action of $G$ on $\mu^{-1}(\chi)$ is free.
Let $S\subset\mathbb{C}$ be a large enough finitely generated subring over which $X, f$ and the action of $G$ on $X$ are defined.
Let $U \subset (\mf{g}/[\mf{g}, \mf{g}])_S^*$ denote the complement of the zero locus of $f.$ Thus, $G_S$ acts freely on $\mu^{-1}(U).$ 
Localizing $S$ further and using the generic flatness theorem, we may assume that 
$\mathcal{O}(\mu^{-1}(0)//G)$ and $\mathcal{O}(\mu^{-1}(0))/\mathcal{O}(\mu^{-1}(0)//G)$ is a flat $S$-module.

Let $e_1,\cdots, e_l$ be a basis of $\mf{g}_{\mathbb{Z}}/[\mf{g}_{\mathbb{Z}}, \mf{g}_{\mathbb{Z}}]$ over $\mathbb{Z}.$
Let $S\to\bold{k}$ be a base change to a characteristic $p$ field $\bold{k}$, let $\bar{\chi}$ denote the image of $\chi$
in $\mathfrak{g}_{\bold{k}}^*$. Then $$\bar{\chi}^{[1]}(\bar{e_i})=(\bar{\chi}(\bar{e}_i))^p-\bar{\chi}(\bar{e}_i).$$

Let $W\subset (\mf{g}/[\mf{g}, \mf{g}])^*$ be the set of all 
$ \chi$ so that $\chi(e_i)$ are algebraically independent over $S.$ Clearly $W$ is a very generic subset.
We will show that for any $\chi\in W$ algebra $U_{\chi}(G, X)$ is simple.

Put $R=U_{\chi}(G, X).$ We verify that $R$ satisfies assumptions in Lemma \ref{criterion}.
Indeed, let $S\to \bf{k}$ be a base change  to an algebraically
closed field $\bf{k}$ of characteristic $\gg 0,$ let $\bar{\chi}$ denote the base change of $\chi.$ 
Recall that $R$ is equipped with the filtration so that $\Gr(R)=\mathcal{O}(\mu^{-1}(0)//G).$
In particular, $R$ is a free $S$-module.
Similarly, $U_{\bar{\chi}}(G_{\bf{k}}, X_{\bf{k}})$ is equipped with the filtration such that 
$\Gr(U_{\bar{\chi}}(G_{\bf{k}}, X_{\bf{k}}))$ is a subring of $\mathcal{O}(\mu_{\bold{k}}^{-1}(0)//G_{\bold{k}}).$
Now applying Theorem \ref{change} to the action of $G$ on $\mathcal{O}(\mu^{-1}(0)),$ we conclude that 
$\mathcal{O}(\mu^{-1}(0)//G)\otimes_S\bf{k}$ surjects onto $\mathcal{O}(\mu_{\bold{k}}^{-1}(0)//G_{\bold{k}}).$
 On the other hand, since $\mathcal{O}(\mu^{-1}(0))/\mathcal{O}(\mu^{-1}(0)//G)$ is flat over $S$, we get that 
$$\mathcal{O}(\mu^{-1}(0)//G)\otimes_S\bold{k}\to \mathcal{O}(\mu^{-1}(0))\otimes_S\bf{k}$$ is injective.
 So, the restriction map $$\mathcal{O}(\mu^{-1}(0)//G)\otimes_S\bold{k}\to \mathcal{O}(\mu_{\bold{k}}^{-1}(0)//G_{\bold{k}})$$
is an isomorphism.
 Therefore, $\Gr(R)\otimes_S{\bold{k}}\cong \Gr(U_{\bar{\chi}}(G_{\bold{k}}, X_{\bold{k}})).$
 Now flatnest of $\Gr(R)$ over $S$ implies that $\Gr(R_{\bold{k}})=\Gr(R)\otimes_S{\bold{k}}.$
 Hence we conclude that $R_{\bold{k}}\cong U_{\bar{\chi}}(G_{\bold{k}}, X_{\bf{k}}).$
 

Since $\chi(e_1),\cdots, \chi(e_n)$ are algebraically independent over $S$, 
we get \\that $\bar{f}(\chi^{[1]})\neq 0$ for all $p\gg 0$ and an
appropriate base change $S\to\bf{k}.$ Hence $\bar{\chi}\in U_{\bf{k}}.$ 
As $G_S$ acts freely on $\mu^{-1}(U),$ 
we conclude that $G_{\bf{k}}$ acts freely on $\mu^{-1}(\bar{\chi}).$ So Lemma \ref{Azumaya} implies that $U_{\bar{\chi}}(G_{\bold{k}}, X_{\bf{k}})$
is an Azumaya algebra over a symplectic variety under the reduction modulo $p$
Poisson bracket. So, conditions of Lemma \ref{criterion} are met. Hence we have shown that algebra $U_{\chi}(G, X)$ is simple
for very generic values of $\chi.$

Now suppose there exists a nonzero $f\in \mathfrak{g}_{\mathbb{Z}}/[\mathfrak{g}_{\mathbb{Z}}, \mathfrak{g}_{\mathbb{Z}}]$  such that $G$ acts freely on $\mu^{-1}(\chi)$
when $\chi(f)\neq 0.$ Let  $S$  be a finitely generated subring containing $\chi(e_i)$ satisfying conditions as above.
Write $f=\sum_if_ie_i, f_i\in \mathbb{Z}.$ Then given a base change $S\to\bold{k}$, 
we have $$\bar{\chi}^{[1]}(\bar{f})=\sum_i\bar{f}_i((\bar{\chi}(\bar{e}_i))^p-\bar{\chi}(\bar{e}_i))=\bar{\chi}(\bar{f})^p-\bar{\chi}(\bar{f}).$$
Let $\chi$ be so that $\chi(f)$ is irrational.
Then it follows from the Chebotarev density theorem that there are
arbitrarily large primes $p$ and a base change $S\to \bf{k}$ to an algebraically closed field
$\bold{k}$ of characteristic $p$, such that $\bar{\chi}(\bar{f})\notin \mathbb{F}_p$. Hence   $\chi^{[1]}(\bar{f})$ is nonzero in $\bold{k}.$
So, $G_{\bf{k}}$ acts freely on $\mu^{-1}(\bar{\chi}),$ and arguing
just as above we may conclude that the algebra $U_{\chi}(G, X)$ is simple.


\end{proof}

We may apply the above result to certain filtered quantizations of quiver varieties as follows. 
Let $Q$ be a quiver with $n$ vertices,
let $\alpha$ be a its positive root. Then $G=\prod GL_{\alpha_i}/\mathbb{C}^*$ acts
on the space of $\alpha$-dimensional representations $Rep(Q, {\alpha})$
giving rise to the moment map $m_{\alpha}:T^*(Rep(Q, {\alpha}))\to \mathfrak{g}^*.$
We will identify $(\mathfrak{g}^*)^G$ with $\lambda\in \mathbb{C}^n$ such that $\lambda\cdot \alpha=0.$
From now on we assume that the moment map $m_{\alpha}$ is flat. The set of such
dimension vectors $\alpha$ was fully described by Crawly-Boevey in [\cite{CB} Theorem 1.1]
Denote by $A_{\lambda}(Q, \alpha)$ the corresponding quantum Hamiltonian reduction
of the ring of differential operators $D(Rep(Q, {\alpha}))$ with respect to the character $\lambda.$

We have the following direct corollary  of Theorem \ref{main}. Remark that stronger results on generic simplicity 
follows from  the works of Losev on
quantizations of quiver varieties (see for example  [\cite{L}, Theorem 1.4.2].)
\begin{theorem}
Let $\alpha$ be a positive root as above.
Let $\lambda\cdot \alpha=0$ be such that $\lambda\cdot\beta\notin \mathbb{Q}$
for any positive root $\beta<\alpha.$ Then $A_{\lambda}(Q, \alpha)$ is simple.

\end{theorem}

\end{document}